\documentclass[11pt]{amsart}
\usepackage{amsmath, amssymb, amsfonts, verbatim, amsthm, xcolor, mathrsfs}
\usepackage{graphics, epsfig, enumerate, amstext,  comment}
\usepackage{graphicx, ulem, hyperref}
\usepackage{bm}
\newtheorem{theorem}{Theorem}[section]
\newtheorem{definition}[theorem]{Definition}
\newtheorem{proposition}[theorem]{Proposition}
\newtheorem{corollary}[theorem]{Corollary}
\newtheorem{lemma}[theorem]{Lemma}
\newtheorem{example}[theorem]{Example}

\def\qed{\hfill $\Box$\medskip}
\def\diag{{\rm diag}\,}
\def\bM{{\mathbb M}}
\def\bB{{\mathbb B}}
\def\bS{{\mathbb S}}
\newcommand{\tens}[1]{\mathbin{\mathop{\otimes}\limits_{#1}}}

\def\Span{\mathop{\rm Span}\nolimits}

\def\rk{\mathop{\rm rank}}

\def\cS{{\mathcal S}}

\def\IC{{\mathbb{C}}}
\def\IR{{\mathbb{R}}}

\def\IL{{\mathbb{L}}}

\def\IT{{\mathbb{T}}}

\def\bV{{\mathbf V}}

\def\tr{\mathop{\rm tr}}

\def\IF{{\mathbb F}}

\def\cl{{\bf cl}}

\def\bL{{\mathbf L}}

\def\Pert{{\rm Pert}}
\def\IC{{\mathbb C}}
\def\IR{{\mathbb R}}
\def\tr{{\mathrm{Tr}}}

\def\cC{{\mathcal C}}
\def\bH{{\mathbb H}}
\begin{document}
\openup 1\jot

\title[Preserving parallel pairs with respect to the Ky-Fan $k$-norm]{Linear maps preserving parallel matrix pairs \\
with respect to the Ky-Fan $k$-norm}
\author{Kuzma, Li, Poon, Singla}
\maketitle

\begin{abstract}
Two bounded linear operators $A$ and $B$ are parallel with respect to a norm
$\|\cdot\|$ if $\|A+\mu B\| = \|A\| + \|B\|$ for some scalar $\mu$ with $|\mu| = 1$.
Characterization is obtained for bijective linear maps sending
parallel bounded linear operators to parallel bounded linear operators with respect to the Ky-Fan $k$-norms.
\end{abstract}


\section{Introduction}

Let $V$ be a normed space equipped with the norm $\| \cdot \|$ over $\IF = \IC$ or $\IR$.
Suppose $x, y \in V$. Then
$x$ is parallel to $y$, denoted by $x\| y$,
if $\|x + \mu y\| = \|x\| + \|y\|$ for  some $\mu\in \IF$ with $|\mu| = 1$.
We are interested in studying bijective
linear maps $T\colon  V \rightarrow V$ preserving parallel pairs,
i.e., $T(x)\| T(y)$ whenever $x\|y$.
Note that a map of the form $A \mapsto f(A)B$ for a fixed $B$
always preserves parallel pairs, so one needs to be careful in weakening
the invertibility assumption.
Evidently, if $T$ is a multiple of a linear  isometry for $\|\cdot\|$, then
$T$ will preserve parallel pairs. However, the converse may not hold for a
general norm.
For instance, suppose $\|\cdot\|$ is a strictly convex norm, i.e.,
$\|x+y\| < \|x\|+ \|y\|$ whenever $x,y$ are linearly independent. Then $(\alpha x,\beta x)$ with $x\in V$ and $\alpha,\beta\in\IF$ are the only
parallel pairs  and consequently  in strictly convex spaces every linear map preserves parallel pairs.

\medskip
Denote by $\bM_n(\IF)$ the linear space of $n\times n$ matrices over the field $\IF$. It turns out that if $\|\cdot\|$ is the operator norm on $\bM_n(\IF)$, then
a bijective linear map $T\colon \bM_n(\IF) \rightarrow \bM_n(\IF)$ preserving parallel
pairs must be a multiple of an isometry except for the $2\times 2$ real case. In the $2\times 2$ case,
there are additional bijective preservers
of parallel pairs. These exceptional maps also preserve matrix pairs $(A,B)$ such that
$\|A+B\| = \|A\| + \|B\|$ in the real case, but
not in the complex case; see \cite{LTWW}. Hence, if $n > 2$, then a bijective linear map preserving parallel pairs with respect to the operator norm has the form
\begin{equation}
\label{standard}
A \mapsto \gamma UAV  \quad \hbox{ or } \quad A \mapsto \gamma UA^tV,
\end{equation}
where $U, V\in \bM_n(\IF)$ are unitary (orthogonal when $\IF = \mathbb R$), and $A^t$ is the transpose of $A$ with respect
to a fixed orthonormal basis and $\gamma$ is a nonzero scalar.

\medskip
Let $k \in \{1, \dots, n\}$. Given $A \in \bM_n(\IF)$, denote by $s_1(A) \ge \cdots \ge s_n(A)$ the singular values of $A$, i.e., the nonnegative square roots of the eigenvalues of $A^*A$. The Ky-Fan $k$-norm for $A \in \bM_n(\IF)$ is defined as $\|A\|_{(k)}=s_1(A)+\dots+s_k(A)$. Evidently, $\|A\|_{(1)}=s_1(A)$
reduces to the operator (i.e., spectral) norm. When $k=n$, $\|A\|_{(n)}$ is known as the trace norm and corresponds to the dual of the spectral norm.

In this paper, we  characterize bijective linear maps preserving parallel pairs with respect to the Ky-Fan $k$-norm on $\bM_n(\IF)$. Since the result for $k = 1$ is known, we will focus on $k > 1$.
We will prove our result for complex matrices in Section \ref{section2}, and for real matrices in Section~\ref{section3}.  Our results can be easily applied to characterize linear bijections which preserve  matrix pairs satisfying  equality for the triangle inequality in a Ky-Fan~$k$-norm.

\section{Results on complex matrices}\label{section2}
In this section, we prove the following.

\begin{theorem} \label{main2}  Let $\|\cdot\|_{(k)}$ be the Ky-Fan $k$-norm on $\bM_n(\IC)$ for $1 < k \le n$. Suppose $T\colon \bM_n(\IC) \rightarrow \bM_n(\IC)$ is a bijective linear map which preserves parallel pairs, i.e., $$A\|B\quad\hbox{ implies }\quad T(A)\| T(B)$$ with respect to $\|\cdot\|_{(k)}$. Then $T$ is a scalar multiple of a linear isometry of $(\bM_n(\IC),\|\cdot\|_{(k)})$. More precisely, there exists $\gamma > 0$ and unitary $U, V \in \bM_n(\IC)$ such that
$T$ has the form $$X \mapsto \gamma UXV \qquad \hbox{ or } \qquad X \mapsto \gamma UX^tV.$$
\end{theorem}
By the above theorem, one can determine the structure of bijective linear maps preserving
matrix pairs $A, B\in \bM_n(\IC)$ satisfying the equality in the triangle inequality, i.e.,
$\|A+B\| = \|A\|+\|B\|$.
Clearly, if a linear bijection $T$ preserves such pairs,
then it also preserves parallelism. To see this,   if $A\|B$, then for a complex unit $\mu$
we have $\|A+\mu B\|=\|A\|+\|\mu B\|$, hence $(A,\mu B)$ is a pair where the triangle inequality becomes equality, so that
$\|T(A)+ \mu T(B)\| = \|T(A)\| + \|\mu T(B)\|$. Let us state this formally.

\begin{corollary}\label{maincor}
 Assume a linear bijection $T\colon\bM_n(\IC)\to\bM_n(\IC)$ is such that $\|A+B\|_{(k)}=\|A\|_{(k)}+\|B\|_{(k)}$ implies $\|T(A)+T(B)\|_{(k)}=\|T(A)\|_{(k)}+\|T(B)\|_{(k)}$, where $\|\cdot\|_{(k)}$ is a Ky-Fan $k$-norm. Then, $T$ is a scalar multiple of a linear isometry of $(\bM_n(\IC),\|\cdot\|_{(k)})$ and hence takes one of the forms from Theorem~\ref{main2}.
\end{corollary}

We will first present the proof of the theorem for the trace-norm (i.e., for $k=n$) and then for general Ky-Fan $k$-norms with $1 < k < n$.
 We start by listing some equivalent conditions for $A, B \in \bM_n(\IC)$ to be parallel relative to the Ky-Fan $k$-norm. We will write
`psd' for a positive semidefinite matrix and we will denote the set of all $n\times n$ psd matrices by $\mathrm{psd}_n$.
A complex unit is a complex number with modulus one.
For notational simplicity, we write~$\bM_n$ instead of $\bM_n(\IC)$ in this section.

\begin{proposition}\label{lem:P(A)}
Let $1< k\le n$. Two matrices $A, B \in \bM_n$ are parallel if and only if any one of the following conditions holds.
\begin{itemize}
\item[{\rm (a)}] There are unitary $U, V \in \bM_n$ and a complex unit $\mu$ such that $U^*(A+ \mu B)V = \diag(c_1, \dots, c_n)$ with $c_1 \ge \cdots \ge c_n$ satisfying
 $\sum_{j=1}^k c_j = \|A\|_{(k)} + \|B\|_{(k)}$.
\item[{\rm (b)}] There are unitary  $U, V \in \bM_n$ and a complex unit $\mu$ such that $U^*AV = A_1 \oplus A_2$ and $U^*BV = B_1 \oplus B_2$, where $A_1 = \diag(s_1(A), \dots, s_k(A))$ and $\mu B_1$ is  psd  with eigenvalues $s_1(B), \dots, s_k(B)$.
\item[{\rm (c)}] There are $n\times k$ matrices $U_1, V_1$ and a complex unit $\mu$ such that $U_1^*U_1 = V_1^*V_1 = I_k$, $U_1^*AV_1 = \diag(s_1(A), \dots, s_k(A))$, and $\mu U_1^*BV_1$ is  psd  with eigenvalues $s_1(B), \dots, s_k(B)$.
\item[{\rm (d)}] There are two orthonormal sets $\{u_1, \dots, u_k\}, \{v_1, \dots, v_k\}  \subseteq \IC^n$ such that $Av_j = s_j(A) u_j$ for $j = 1, \dots, k$ and $[u_1 |\cdots| u_k]^*B[v_1 |\cdots | v_k]$ is a unit multiple of a psd matrix with eigenvalues $s_1(B), \dots, s_k(B)$.
\end{itemize}

\end{proposition}

\begin{proof} Condition (a) is clearly equivalent to $A\| B$. Suppose (a) holds. Let $U^*AV$ and $U^*BV$ have leading $k\times k$ submatrices $A_1$ and $B_1$. Then, by the assumptions in (a) and by \cite[Lemma~2]{Thompson}, $$\|A\|_{(k)} + \|B\|_{(k)} = |\tr (A_1+\mu B_1)| \le |\tr A_1| + |\tr B_1| \le \|A\|_{(k)} + \|B\|_{(k)}.$$
Thus, $|\tr A_1| = \sum_{j=1}^k s_j(A)$ and $|\tr B_1| = \sum_{j=1}^k s_j(B)$. By  \cite[Corollary 3.2]{Li} there are complex units $\mu_1, \mu_2$ such that $U^*AV = A_1 \oplus A_2$ and $U^*BV = B_1 \oplus B_2$, where $\mu_1 A_1$  is  psd  with eigenvalues
$s_1(A), \dots, s_k(A)$ and $\mu_2 B_1$ is  psd  with eigenvalues $s_1(B), \dots, s_k(B)$, respectively. Let $W \in \bM_k$ be a unitary with $\mu_1 W^*A_1 W = \diag(s_1(A), \dots, s_k(A))$. We get condition (b) if we replace $(U,V)$ by $(U(W \oplus I_{n-k}), \mu_1 V(W\oplus I_{n-k}))$ and then let $\mu = \bar \mu_1 \mu_2$.

If (b) holds, then we  have $$\|A+  \mu B\|_{(k)} = \|U^*(A+ \mu B)V\|_{(k)} \ge |\tr A_1+ \mu \tr B_1| = \|A\|_{(k)} + \|B\|_{(k)}.$$ Thus, $A$ and $B$ are parallel. The equivalence of (b), (c), (d)  is clear.
\end{proof}
We remark that  we will rely heavily on the equivalence between (a) and (b) of the above  proposition in our subsequent proofs. The other equivalences might be of independent interest. 

A key step in our proofs is to show that a bijective linear map preserving parallel pairs
has the form $T\colon X\mapsto MXN$ and $T\colon X\mapsto MX^tN$ for invertible matrices $M$ and $N$. One can then finish the proof using the following lemma.

\begin{lemma}\label{lem:rank-one=preservers}
Let $M,N\in\bM_n$ be invertible matrices. The maps $T\colon X\mapsto MXN$ and $T\colon X\mapsto MX^tN$ preserve parallel pairs with respect to a Ky-Fan $k$-norm $(1< k\le n)$ if and only if both $M$ and $N$ are scalar multiples of a unitary matrix.
\end{lemma}
\begin{proof} Let $M = X_1D_1Y_1$ and $N = X_2 D_2 Y_2$, where $X_1, X_2, Y_1, Y_2 \in \bM_n$ are unitary, and $D_1 = \diag (\xi_1, \dots, \xi_n)$ and $D_2 = \diag (\eta_1, \dots, \eta_n)$
with $\xi_1 \ge \cdots \ge \xi_n > 0$ and $\eta_1 \ge \cdots \ge \eta_n > 0$.

We can replace $T$ by the linear map $\Psi$ defined by $$A\mapsto [X_1^* T(Y_1^*AX_2^*)Y_2^*]/(\xi_1\eta_1),$$ or by $$A\mapsto [X_1^*T((Y_1^*AX_2^*)^t)Y_2^*]/(\xi_1\eta_1)$$ if the transpose map is involved. Then $\Psi$ will also preserve parallel pairs and has the simple form $A \mapsto \hat D_1A \hat D_2$, where $\hat D_1= D_1/\xi_1$, $\hat D_2 = D_2/\eta_1$ have diagonal entries $1, \hat \xi_2, \dots, \hat \xi_n$ and $1, \hat \eta_2, \dots, \hat\eta_n$. Observe that $G_1=e_1e_1^\ast$ is parallel to $G_2=(e_1+e_2)(e_1+e_2)^\ast$ because they are both rank-one psd matrices. Hence, $\Psi(G_1)=e_1e_1^\ast$ and $\Psi(G_2)=(e_1+\hat{\xi}_2e_2)(e_1+\hat{\eta}_2e_2)^\ast$ are also parallel and both belong to a  subspace $\bM_2\oplus 0_{n-2}$. Since parallelism is defined  by the norm, $\Psi(G_1),\Psi(G_2)$ are parallel if and only if their compressions to~$\bM_2$ are parallel.  By Proposition~\ref{lem:P(A)}(b), there exist unitary $U_1,V_1\in\bM_2$ such that $U_1\Psi(G_1)V_1=U_1e_1e_1^\ast V_1=e_1e_1^*$ and $U_1\Psi(G_2)V_1\in\bM_2$ is a scalar multiple of a psd. From the first identity, $U_1,V_1$ are both diagonal unitary matrices, and then the second requirement, together with $\hat{\eta}_2,\hat{\xi}_2>0$,  implies $\hat{\eta}_2=\hat{\xi}_2$. Similarly, $\hat{\eta}_i=\hat{\xi}_i$ for each $i$ so
$$\hat D_1=\hat D_2=:D=\diag(1,\hat{\eta}_2,\dots,\hat{\eta}_n).$$

Consider next $G_1=e_1(e_1+e_2)^\ast$ and $G_2=e_2(-e_1+e_2)^\ast$ which are also parallel because, for $W = \left(
\begin{smallmatrix}
	\frac{1}{\sqrt{2}} & -\frac{1}{\sqrt{2}} \\
	\frac{1}{\sqrt{2}} & \frac{1}{\sqrt{2}} \\
\end{smallmatrix}
\right) \oplus I_{n-2},$
$G_1 W$ and $G_2 W$ are both rank-one psd matrices.
They are mapped into the parallel pair
$$ \Psi(G_1)=DG_1D=e_1(e_1+\hat{\eta}_2e_2)^\ast
\quad\hbox{and}\quad \Psi(G_2)=DG_2D=\hat{\eta}_2e_2(-e_1+\hat{\eta}_2e_2)^\ast
.$$
Again we may consider their parallelism restricted to $\IR^2$. Up to an inessential scaling,
the only possible unitary matrices $U,V\in\bM_2$ for which $U\Psi(G_1)V=\diag(\ast,0)\in\bM_2$ are  $U=\left(\begin{smallmatrix} 1 & 0 \\ 0 & \epsilon \end{smallmatrix}\right)$ and
$V = \frac{1}{\sqrt{\hat{\eta}_2 ^2+1}} \left( \begin{smallmatrix} 1 & \delta \hat{\eta}_2 \\ \hat{\eta}_2 & - \delta \end{smallmatrix} \right)$
with $|\epsilon|=|\delta|=1$. However, with these unitary matrices, the matrix $U \Psi(G_2)V$ is a scalar multiple of a psd if and only if $|\hat{\eta}_2|=1$. Likewise we can argue to show that $|\hat{\eta}_i|=1$. Hence, $D = I_n$, so $M$ and $N$ are indeed scalar multiples of unitary matrices.
\end{proof}

Recall that a cone $\mathcal{C}$ in $\bM_n(\IC)$ is a convex set which is closed under multiplication with nonnegative real numbers. Observe that we assume that cones are
pointed, i.e. they contain 0.
A cone (of affine dimension $d$) is maximal if it is not properly contained in a larger cone (of affine dimension~$d$). For a given $A \in \bM_{n}$, we shall make heavy use of the set $$P(A) = \{B \in \bM_n: B \hbox{ is parallel to } A\}.$$   One important observation is that if $T\colon  \bM_n \rightarrow \bM_n$ is a bijective linear map such that $T(A)$ is parallel to $T(B)$ whenever $A$ is parallel to $B$, then we have $$T(P(A)) \subseteq P(T(A)), \quad \hbox{and hence } \quad T(\Span P(A)) \subseteq \Span P(T(A)).$$

\subsection{The trace norm}

We start with a characterization of a singular matrix $A$ in terms of existence of a particular convex cone in $P(A)$ with respect to the trace norm.

\begin{lemma}\label{lem:trace-norm-cone-classification-of-singular}
Let $\|\cdot\|$ be the trace norm on $\bM_n$. Then $A \in \bM_n$ is singular if and only if there is a convex cone ${\mathcal C}$  such that  $\{A, B, -B\} \subseteq {\mathcal C} \subseteq P(A)$ for some $B$ not equal to a  multiple of $A$.
\end{lemma}

\begin{proof} Let $A = U\diag(a_1, \dots a_n)V$ be the singular value decomposition, with $a_1\ge\dots\ge a_n\ge0$. If $A$ is singular, then $a_n = 0$. Let $B = UE_{nn}V$. Then
${\mathcal C} = \{t_1 A + t_2 B + t_3(-B):  t_1, t_2, t_3 \ge 0\}$ is the desired cone.

If $A$ is invertible then $a_n \ne 0$. Assume to the contrary that the cone ${\mathcal  C}$ and matrix $B$ do exist. Since $U^\ast P(A)V^\ast=P(U^\ast AV^\ast)$ we may apply the bijection $X\mapsto U^\ast XV^\ast$ and assume from the start that $A=\diag(a_1,\dots,a_n)$. Then $tA + B \in{\mathcal C}\subseteq P(A)$ for all $t \ge 0$, so by Proposition \ref{lem:P(A)} (b), there are unitary $X_t,Y_t$ and a complex unit $e^{i\phi_t}$ with
\begin{equation}\label{eq:trace-norm}
X_t AY_t=A\quad\hbox{ and }\quad X_t(tA + B)Y_t\in e^{i\phi_t}\mathrm{psd}_n,\qquad t\ge0.
\end{equation}
By grouping equal eigenvalues in $A$ and considering $Y_t^\ast A^\ast A Y_t=(X_tAY_t)^\ast(X_tAY_t)=A^\ast A$, we see that $Y_t$, and similarly $X_t$, inherit the block-diagonal structure  $Y_t=Y_{1,t}\oplus\dots\oplus Y_{j,t}$ and $X_t=X_{1,t}\oplus\dots\oplus X_{j,t}$ with block sizes equal to the multiplicities of the corresponding eigenvalues in $A$. Moreover,  since the eigenvalues of $A$ are positive,  $X_{k,t}=Y_{k,t}^\ast$ and hence $X_t=Y_t^\ast$.  Using this in~\eqref{eq:trace-norm} we deduce that $A^{1/2}(t I_n+A^{-1/2}BA^{-1/2})A^{1/2}=tA + B\in  e^{i\phi_t}\mathrm{psd}_n$, and hence also $$t I_n+A^{-1/2}BA^{-1/2}\in e^{i\phi_t}\mathrm{psd}_n;\qquad t\ge0.$$ At $t=0$ we get $A^{-1/2}BA^{-1/2}=e^{i\phi} P$ for some psd matrix $P$ and hence $$tI_n+e^{i\phi} P\in e^{i\phi_t} \mathrm{psd}_n.$$ Thus for each $t \geq 0$, the nonzero eigenvalues of $tI_n+e^{i\phi} P$ all have the same argument.  Since $B$ is not a scalar multiple of $A$, $P \geq 0$ is not a scalar matrix.  Let $0 \leq \lambda_1 < \lambda_2$ be eigenvalues of $P$ and choose $t \in (\lambda_1, \lambda_2)$.  Then $t +e^{i\phi} \lambda_1$ and  $t +e^{i\phi} \lambda_2$ have the same argument, so $e^{i\phi} = 1$. Consequently $B$ is psd. But applying the same reasoning to $-B$, we see  that $-B$ is psd, which is a contradiction. Indeed, if $A$ is invertible, such a cone ${\mathcal C}$ cannot exist.
\end{proof}

\begin{proof}[Proof of Theorem \ref{main2} for the trace norm]
We first show that $T$ maps singular matrices into singular ones. If $A \in \bM_n$ is singular, there is a cone ${\mathcal C}$ of the form described in Lemma~\ref{lem:trace-norm-cone-classification-of-singular}. So,  $\{T(A), T(B), -T(B)\} \subseteq T({\mathcal C}) \subseteq T(P(A))\subseteq P(T(A))$. Since $T({\mathcal C})$ is a cone, $T(A)$ is singular again by Lemma~\ref{lem:trace-norm-cone-classification-of-singular}. Thus $T$ is a bijective map sending singular matrices to singular matrices. By  Dieudonn\'{e}'s result~\cite{D},
$T$ has the standard form $A\mapsto MAN$ or $A \mapsto MA^tN$ for some invertible $M,N \in \bM_n$. By Lemma~\ref{lem:rank-one=preservers} they are both multiples of a unitary matrix.\end{proof}

\subsection{The Ky-Fan norm for $1 < k < n$} The proof for the case when $1 < k < n$ is a bit involved.
Let $T\colon \bM_n\rightarrow \bM_n$ be a bijective linear map preserving parallel pairs relative to $\|\cdot\|_{(k)}$. We give the outline of the proof in various steps below, which will be proved subsequently.

{\bf Step 1.} We will show that $\Span P(A)$ (i.e., the complex-linear span of $P(A)$) has a minimum dimension equal to $k^2 + (n-k)^2$, and equality holds if and only if $s_k(A) > s_{k+1}(A)$.  These conditions hold if and only if there are unitary $U, V$ such that $U (\Span P(A)) V = \bM_k\oplus \bM_{n-k}$. Consequently, for such an $A$, $X= T^{-1}(A)$ must also satisfy $s_k(X) > s_{k+1}(X)$, and there are unitary $P, Q$ such that $P(\Span P(X)) Q = \bM_k \oplus \bM_{n-k}$.

{\bf Step 2.} We use the results in Step 1 to analyze the behaviour of the set $$\cC = \{X_1 \oplus X_2\in\bM_k\oplus\bM_{n-k}: X_1 \hbox{ is a positive semidefinite matrix with } s_k(X_1) \ge  s_1(X_2)\}$$ under $T$. It turns out that $\cC$ is a cone and under the map $T$, the relative interior (relative boundary) of $\cC$ are mapped exactly to the relative interior (relative boundary).

{\bf Step 3.} Another important set useful in our analysis will be $$\Pert(X) = \{Z \in \partial \cC : X \pm tZ \in \cC \text{ for sufficiently small } t>0 \}.$$  We will analyze when the real dimension of $\Span_{\IR} (\Pert(X))$ (i.e., the real-linear span of $\Pert(X)$) equals one; this will lead us to prove that the rank one matrices in $\partial \cC$ are mapped exactly to rank one matrices.

{\bf Step 4.} We will prove that $T$ sends rank one matrices to rank one matrices and use the standard results to conclude Theorem \ref{main2}.

We start by counting the dimension of the span of $P(A)$ subject to different choices for the matrix~$A$. In part (b)  we will encounter matrices of the form $W = I_p \oplus W_1 \oplus I_{n-q} \in \bM_p \oplus \bM_{q-p} \oplus \bM_{n-q}$ where $0 \le p < q \le n$. At the extremal cases when $p=0$ or when $q=n$  we agree that $I_0$ and $\bM_0$ indicate that the  corresponding summand is omitted.

\begin{lemma}\label{Lem:dimP(A)} Let $A \in \mathbb \bM_n(\IC)$ and let $U, V \in\bM_n(\IC)$ be unitary matrices such that $U^*AV = \diag(s_1(A), \dots, s_n(A))$. 
Then, the following holds.
\begin{itemize}
\item[{\rm (a)}] Suppose $s_k(A)> s_{k+1}(A)$. Then $B \in P(A)$ if and only if $U^*BV = B_1 \oplus B_2$ such that $\mu B_1 \in \bM_k(\IC)$ is psd for some complex unit $\mu$ and $|\tr B_1| = \|B\|_{(k)}$. Moreover, $$ \Span P(A) =  U(\bM_k(\IC) \oplus \bM_{n-k}(\IC))V^*.$$ Consequently, $\Span P(A)$ has dimension $k^2 + (n-k)^2 = n^2 -2kn + 2k^2$.
\item[{\rm (b)}]  Suppose $s_k(A) = s_{k+1}(A)$. Let $p\ge0$ and $q\le n$ be the smallest integer
and the largest  integer such that $s_{p+1}(A) = s_k(A) = s_q(A)$.
\begin{itemize}
\item If $q \ne n$ or $A$ is invertible, then $B \in P(A)$ if and only if there is a unitary $W$
of the form $W = I_p \oplus W_1 \oplus I_{n-q} \in \bM_p(\IC) \oplus \bM_{q-p}(\IC) \oplus \bM_{n-q}(\IC)$ such that $B = U W (B_1 \oplus B_2) W^* V^*$, $\mu B_1\in \bM_k(\IC)$ is psd for some complex unit $\mu$, and $\|B_1\|_{(k)} = \|B\|_{(k)}$.
\item If $q=n$ and $A$ is singular, then $B \in P(A)$ if and only if there exist unitary $W_2, W_3 \in \bM_{n-p}(\IC)$ such that $B = U (I_p \oplus W_2) (B_1 \oplus B_2) (I_p \oplus W_3) V^*$, $\mu B_1\in \bM_k(\IC)$ is psd for some complex unit $\mu$, and $\|B_1\|_{(k)} = \|B\|_{(k)}$.
\end{itemize}
 Consequently, $\Span P(A)$ has dimension at least $k^2 + 2k(q-k)+(n-k)^2$.
\end{itemize}
\end{lemma}
\begin{proof}
Since $P(XAY) = XP(A)Y$ for all unitaries $X,Y \in \bM_n$, we may assume that $A$ is the diagonal matrix $\diag(s_1(A), \dots, s_n(A))$. Let $\{e_1, \dots, e_n\}$ be the standard basis for $\IC^n$. Then $Ae_j  = s_j(A) e_j$ for $j = 1, \dots, n$.

(a) Suppose $s_k(A) > s_{k+1}(A)$. Let $B \in P(A)$. By Proposition \ref{lem:P(A)}, there are unitary $U, V \in \bM_n$ and a complex unit $\mu$ such that $U^*AV = A_1 \oplus A_2$ with $A_1 = \diag(s_1(A), \dots, s_k(A))$ and $U^*BV = B_1 \oplus B_2$ such that
$\mu B_1 \in \bM_k$ is  psd with eigenvalues $s_1(B), \dots, s_k(B)$. Since $s_k(A) > s_{k+1}(A)$,  the first $k$ columns of $V$ are eigenvectors corresponding to the $k$ largest eigenvalues of $A^*A$. So, the first $k$ columns of $V$ lie in the span of $\{e_1, \dots, e_k\}$. Thus, $V = V_1 \oplus V_2$ with $V_{1} \in \bM_k$. Similarly, we can show that $U = U_{1} \oplus U_{2}$ with $U_{1} \in \bM_k$. Then $U_1^* A_1 V_1 = A_1$, so $V_1^* A_1^* A_1 V_1 = A_1^* A_1$.  Thus $V_1$ commutes with
$A_1^2$, so $V_1$ commutes with $A_1$ and $A_1 = U_1^* A_1 V_1 = U_1^* V_1 A_1.$  Since $A_1$ is invertible, $U_1 = V_1$.

Hence $B = (U_{1}B_1 V_1^*)\oplus (U_{2} B_2 V_{2}^*)$. So, every matrix $B$ in $P(A)$ has the form  $\hat B_1 \oplus \hat B_2$, where $\hat B_1$ is a scalar multiple of a psd. Hence $\Span P(A) \subseteq \bM_k \oplus \bM_{n-k}$. Conversely, note that every  psd  matrix $B = B_1 \oplus B_2$ satisfying $s_k(B_1) \ge s_1(B_2)$ lies in $P(A)$. The  span of such matrices is $\bM_k \oplus \bM_{n-k}$, so (a) follows.

(b)  Suppose $s_k(A) = s_{k+1}(A)$ and $p, q$ satisfy the hypothesis. By Proposition \ref{lem:P(A)}, we have $B \in P(A)$ if and only if there are unitary $X,Y$ such that $X^\ast AY=A=\diag(s_1(A),\dots,s_n(A))$ and $X^\ast BY=B_1\oplus B_2\in \bM_k\oplus \bM_{n-k}$, with  $\mu B_1$ psd for some complex unit  $\mu$ and $\|B_1\|_{(k)}=\|B\|_{(k)}$.   By grouping the same  singular values in $A$ and considering $Y^\ast A^\ast AY=(X^\ast AY)^\ast(X^\ast AY)=A^\ast A$ we see that  $Y$, and similarly $X$, inherit the  block-diagonal structure $X=X_1\oplus \dots \oplus X_t$ and $Y=Y_1\oplus \dots\oplus Y_t$ with block sizes equal to the multiplicities of the corresponding singular values in $A$. Moreover,  if the singular value for the $j$th group is nonzero, then $X_j=Y_j$.

Therefore  $B\in P(A)$ if and only if $$B=(X_1\oplus X_2\oplus\dots \oplus X_t)(\mu B_1\oplus B_2)(X_1\oplus X_2\oplus\dots\oplus Y_t)^\ast$$ for some complex unit $\mu$ and psd $B_1\in\bM_k$ with $\|B_1\|_{(k)}=\|B\|_{(k)}$, where $X_j, Y_j$ are all unitary, and $Y_t=X_t$ if  $s_n(A)>0$, i.e., if $A$ is invertible.

Let us index $m \geq 1$ be such that $X_1 \oplus \dots \oplus X_m$ is of size $q$.
Let $W = I_p \oplus X_m \oplus I_{n-q}$, where $I_0$ indicates that that summand is missing.  Let $\hat{X}_j = X_j$ for $j \ne m$ and $\hat{X}_m = I_{q-p}$.  Then
if $q \ne n$, $$B = W (\oplus_{j=1}^{t-1} \hat{X}_j \oplus X_t) (\mu B_1 \oplus B_2) (\oplus_{j=1}^{t-1} \hat{X}_j \oplus Y_t)^* W^* = W (\mu \hat{B}_1 \oplus \hat{B}_2) W^*,$$
where $\hat{B}_1 \in \bM_k$ is psd with $\|\hat{B}_1\|_{(k)} = \|B\|_{(k)}$.  When $q=n$ the above still holds if $A$ is invertible, but if $A$ is singular with $q=n$ then
$$B = (I_p \oplus X_m) (\mu \hat{B}_1 \oplus \hat{B}_2) (I_p \oplus Y_m)$$
for some unitary $X_m, Y_m \in \bM_{n-p}$.

For the assertion about $\Span P(A)$, let $x \in \IC^q$ be a unit vector and let $B= xx^* \oplus 0_{n-q}$.  There exists a unitary $W_1 \in \bM_{q-p}$ such that $(I_p \oplus W_1)x \in \IC^{p+1} \oplus 0_{q-p-1} \subseteq\IC^k\oplus0_{q-k}$.  Letting $W = I_p \oplus W_1 \oplus I_{n-q}$ we see that $WBW^* = yy^* \oplus 0_{n-k}$ for some unit vector $y \in \IC^k$, so $B \in P(A)$.  Thus $\Span P(A)$ contains $\bM_q \oplus 0_{n-q}$.  Since $t I_k \oplus B_2 \in P(A)$ for any $B_2 \in \bM_{n-k}$ if $t > 0$ is sufficiently large, $\Span P(A)$ also contains $0_k \oplus \bM_{n-k}$.  Thus $\dim \Span P(A)$ is at least $k^2 + 2k(q-k) + (n-k)^2$.
\end{proof}

As a corollary, we get an equivalent condition for  $\Span P(A) = \Span P(B)$ for $A, B \in \bM_n$ with $s_k(A) > s_{k+1}(A)$.
\begin{corollary} \label{cor:SASB}
Let $A, B \in \bM_n(\IC)$ with $s_k(A) > s_{k+1}(A)$, and $A = A_1 \oplus A_2$
with $A_1 \in \bM_k(\IC)$ such that  $\sum_{j=1}^k s_j(A_1) = \|A\|_{(k)}$. Then,
$$\Span P(A) = \Span P(B)$$ if and only if $s_k(B) > s_{k+1}(B)$
and $B =  B_1 \oplus B_2$ is such that

{\rm (i)} $B_1 \in \bM_k(\IC)$ satisfies $\sum_{j=1}^ks_j(B_1) =\|B\|_{(k)}$,
 or

 {\rm (ii)} $n = 2k$,  and $B_2 \in \bM_k(\IC)$ satisfies $\sum_{j=1}^ks_j(B_2) =\|B\|_{(k)}$.
\end{corollary}

\begin{proof}
Let $U = U_1 \oplus U_2, V= V_1 \oplus V_2 \in \bM_k \oplus \bM_{n-k}$ be unitary matrices such that $U^*AV = \diag (s_1(A), \dots, s_n(A))$. By Lemma~\ref{Lem:dimP(A)}, $\Span P(A)$ equals $\bM_k \oplus \bM_{n-k}$. If $\Span P(A) = \Span P(B)$ then both spans have the minimal possible dimension, so $s_k(B) > s_{k+1}(B)$.

Let $X, Y$ be unitaries such that $X^*BY = \diag (s_1(B), \dots, s_n(B))$.  By Lemma~\ref{Lem:dimP(A)}, $X(\bM_k\oplus \bM_{n-k})Y^* = \Span P(B) = \bM_k \oplus \bM_{n-k}$. It follows that both $X, Y$ are block-diagonal unitaries, or $n=2k$ and $JX, JY$ are block-diagonal unitaries, where
$J = \left[ \begin{smallmatrix} 0_k & I_k \\ I_k & 0_k \end{smallmatrix} \right]$.  Thus $B = B_1 \oplus B_2$ and either condition (i) or (ii) holds.
The converse follows readily from Lemma~\ref{Lem:dimP(A)}.
\end{proof}

\begin{corollary}\label{cor:minimal}
Suppose a bijective linear $T\colon \bM_n\to\bM_n$ preserves parallel pairs relative to Ky-Fan $k$-norm for $1< k< n$. If $T(X)$ is such that $s_{k}(T(X))>s_{k+1}(T(X))$, then $s_{k}(X)>s_{k+1}(X)$ and $T(\Span P(X))= \Span P(T(X))$.
\end{corollary}
\begin{proof} Since a preserver $T$ satisfies $T(P(X))\subseteq P(T(X))$, and $T$ is bijective, one has $$\dim \Span P(X) = \dim  T(\Span P(X))\le \dim \Span P(T(X)).$$  Lemma~\ref{Lem:dimP(A)} shows that $\Span P(T(X))$ has the minimal possible dimension, and hence so does $\Span P(X)$, whence $s_k(X) > s_{k+1}(X)$, and consequently $T(\Span P(X))= \Span P(T(X))$.\end{proof}	

Now, we proceed towards proving that a bijective linear map which preserves parallel pairs relative to a Ky-Fan $k$-norm for $1<k< n$, is bijective on a subspace $\bM_k \oplus \bM_{n-k}$, and will map the cone $\cC$ onto itself, where $\cC$ is defined as follows.
\begin{definition}\label{defn1} Define $$\cC = \{X_1 \oplus X_2\in\bM_k\oplus\bM_{n-k}: X_1 \hbox{ is a positive semidefinite matrix with } s_k(X_1) \ge  s_1(X_2)\}$$
 \end{definition}

Recall that for a subset $S\subseteq\mathbb M_n$, we  use $\Span(S)$ for its complex span  
and $\Span_{\IR}(S)$ for its  real span. 
\begin{proposition}\label{prop1} The following holds.
\begin{itemize}
\item[(a)] The set $\cC$ is a pointed convex cone.
\item[(b)] Its affine hull is $\Span_{\IR}\cC=\bH_k\oplus\bM_{n-k}$, where $\bH_k$ denotes the set of all $k\times k$ hermitian matrices.
\item[(c)] The  relative interior, $\cC^\circ$, of $\cC$ consists of matrices of the form $X_1 \oplus X_2$ such that $X_1$ is positive definite and $s_k(X_1) > s_1(X_2)$.
\item[(d)] The relative boundary, $\partial\cC$, of $\cC$ consists of matrices of the form $X_1 \oplus X_2$ such that $X_1$ is positive semi-definite and $s_k(X_1) = s_1(X_2)$.
\item[(e)]  We have $\cC\subseteq P(A)$ whenever $A = A_1 \oplus A_2 \in \bM_k \oplus \bM_{n-k}$ is such that $A_1$ is positive definite and $s_k(A_1) > s_1(A_2)$.
 \end{itemize}

\end{proposition}

\begin{proof}
(a). The set $\cC$ is clearly closed under multiiplication by positive numbers. It is also convex because  $X=X_1\oplus X_2\in\cC$ is equivalent to $X_1 \geq 0$ and $\|X\|_{(k)}=\mathrm{Tr} X_1$. So, if $Y=Y_1\oplus Y_2$ is also  in $\cC$, then  $\|X+Y\|_{(k)}\le\|X\|_{(k)}+\|Y\|_{(k)}=\mathrm{Tr}(X_1)+\mathrm{Tr}(Y_1) =\mathrm{Tr}(X_1+Y_1)=\sum_1^k s_j(X_1+Y_1)\le\|X+Y\|_{(k)}$.

(b)---(d) are straightforward exercises, while (e) follows from Proposition~\ref{lem:P(A)}(b).
\end{proof}

\begin{proposition}\label{prop:phi-preservec-cone-CC}
 Let  $T\colon\bM_n\to\bM_n$ be  a bijective linear map which preserves parallel pairs relative to a Ky-Fan $k$-norm for $1<k< n$. Assume there exist $A,B\in{\cC}^\circ$ such that $T(A)=B$. Then $$T({\cC}^\circ) = {\cC}^\circ\quad\hbox{  and }\quad T(\partial \cC) = \partial \cC.$$ Moreover, $T$ maps  $\Span P(B)=\Span P(A)$ bijectively onto itself.
\end{proposition}

\begin{proof} Since matrices $A,B$ belong to the interior of the cone $\cC$, they both  satisfy the assumption~(a) of Lemma~\ref{Lem:dimP(A)}, so that $\bV:=\Span P(B)=\Span P(A)= \bM_k \oplus \bM_{n-k}$ and by Corollary~\ref{cor:minimal} $T(\bV) = \bV$. It remains to prove the first claim.

{\bf Step 1}: We show that $T$ maps $\cC^{\circ}$ into $\cC$. \smallskip

If $X = X_1 \oplus X_2$ belongs to  the interior of $\cC$,
then we claim that $T(X)  = Y =  Y_1 \oplus Y_2$ is in $\cC$.
To see this, observe that  $A + tX$ is parallel to $A$ for all $t > 0$.
So, $T(A) + tT(X) =B + tY$ is parallel to $B = B_1 \oplus B_2$.
Hence,
 $B_1 + tY_1$ is a multiple of a psd  matrix
 and $s_k(B_1+tY_1) \ge s_1(B_2+tY_2)$  for all $t > 0$. Letting $t\to \infty$ we ge  $s_k(Y_1) \geq s_1(Y_2)$ and $Y_1=e^{i\phi} P$ for some psd matrix $P$. Therefore
  $$B_1+te^{i\phi} P=B_1^{1/2}(I_k+te^{i\phi} B_1^{-1/2}PB_1^{-1/2})B_1^{1/2}\in e^{i\varphi_t} \mathrm{psd}_k;\qquad\varphi_t\in[0,2\pi].$$   Clearly then, by defining $\hat{P}=B_1^{-1/2}PB_1^{-1/2}$, also  $$I_k+te^{i\phi} \hat{P}\in e^{i\varphi_t} \mathrm{psd}_k;\qquad\varphi_t\in[0,2\pi],\;t\ge0.$$ Thus for each $t \geq 0$, the nonzero eigenvalues of $I_k+te^{i\phi} \hat{P}$ all have the same argument.  If $\hat{P} \geq 0$ is not a scalar matrix, then let $0 \leq \lambda_1 < \lambda_2$ be distinct eigenvalues of $\hat{P}$ and choose $t \in (\frac{1}{\lambda_2}, \frac{1}{\lambda_1})$. Then $1 +te^{i\phi} \lambda_1$,  $1 +te^{i\phi} \lambda_2$ have the same argument, so $e^{i\phi} = 1$ and $Y_1=P$. Consequently either $Y_1$ is psd or else $\hat{P}=\lambda I_k\ge0$ and  $Y_1=\alpha B_1$ for some  number $\alpha\in\IC\setminus[0,\infty)$.

 Assume there exist $X,X'\in\cC^\circ$ with $T(X)=\alpha B_1\oplus Y_2$ and $T(X')=Y_1'\oplus Y_2'$ such that $\alpha B_1$ is not psd, while $Y_1'$ is psd but not a scalar multiple of $B_1$. Being a convex cone, the matrix line segment $[X,X']:=\{t X+(1-t)X':0\le t\le 1\}$ lies in $\cC^\circ$ for every $t \in[0,1]$, so the first block of its $T$-image,
  \begin{equation}\label{eq:complex}
    t (\alpha B_1) + (1-t) Y_1'
  \end{equation}
is either a psd or a multiple of $B_1$ for each fixed $t\in[0,1]$. The latter option is possible only for $t=1$. But the former option is contradictory because $\alpha\not\ge0$ so the diagonal entries of \eqref{eq:complex} cannot  always be nonegative for each $t \in (0,1)$.

This shows that either $T$ maps $\cC^\circ$ into $\cC$ or else it maps $\cC^\circ$ into $\IC B_1\oplus \bM_{n-k}$. The latter case is impossible because $\Span\cC^\circ=\bM_k\oplus\bM_{n-k}$ has greater dimension that $\IC B_1 \oplus \bM_{n-k}$.

\medskip
{\bf Step 2}: We show that $T(\cC^{\circ}) = \cC^{\circ}$. Consequently, $T^{-1}$ maps $\cC^{\circ}$ into $\cC$.
\smallskip

\medskip

Choose $Y\in\cC^\circ$. By Corollary~\ref{cor:SASB}, $\Span P(Y) = \Span P(I_k\oplus 0_{n-k})=\bV$ and, by Corollary~\ref{cor:minimal}, $X = T^{-1}(Y)$ satisfies $s_k(X) > s_{k+1}(X)$ and \begin{equation}\label{eq:SpanP(X)}
  \Span P(X) =T^{-1}\Span P(Y)=T^{-1}\Span P(B)= \Span P(A).
\end{equation}

So, $X \in \Span P(A)=\bV$, and hence $X = X_1 \oplus X_2\in \bM_k \oplus \bM_{n-k}$. We claim that the $k$ largest singular values of $X=X_1\oplus X_2$ belong to its first block.

{ \sc Suppose  $ {n \ne 2k}$}.  Then the claim follows  by \eqref{eq:SpanP(X)} and Corollary~\ref{cor:SASB}.

\medskip
{\sc Suppose  ${n = 2k}$}. Since  $\cC^{\circ}$ is a cone, the line segment $t B+(1-t) Y \in \cC^{\circ}$ for each $t\in [0,1]$. Hence, as in \eqref{eq:SpanP(X)}, its preimage, $L(t)=T^{-1}(t B+(1-t) Y)=t A+(1-t) X\in\bM_k\oplus\bM_k$ consists of matrices with $s_k(L(t))>s_{k+1}(L(t))$ for each $t$. Moreover, by Corollary~\ref{cor:SASB}, for each $t$
either the first~$k$ singular values of $L(t)$ belong to the first block or they belong to the second block. By continuity and the connectedness of the line segment $[0,1]$,  they either belong to the first block for each $t\in[0,1]$ or they belong to the second block for each $t\in[0,1]$. At  $t=1$, the line segment equals $L(1)=A$ whose first $k$ singular values are in the first block. Hence, also at $t=0$ they are in the first block, that is, $s_k(X_1)>s_1(X_2)$.

Let $X_1=U_1 D_1V_1$ be its singular value decomposition, let $U=U_1\oplus I_{n-k}$ and  $V=V_1\oplus I_{n-k}$ (so that $U(D_1\oplus X_2)V=X$) and temporarily replace $T$ by a map $\hat{T}\colon Z\mapsto T(UZV)$. Observe that $\hat{T}$ still preserves parallelism, leaves the set $\bV$ invariant, and maps $D_1\oplus X_2$ into $T(X) = Y \in \cC^{\circ}$. By Step 1,
$\hat{T}$ maps the cone $\cC$ into itself. Then, $T$ maps the cones $\cC$ and $U\cC V$ into $\cC$. Hence, $T$ also maps their real-linear spans, $$\Span_{\IR}(\cC\cup U\cC V)=\Span_{\IR}(\cC)+U\Span_{\IR}(\cC)V$$ into $\Span_{\IR}(\cC)=\bH_k\oplus\bM_{n-k}$. Observe that \begin{align*}
  U\Span_{\IR}(\cC)V&=(U_1\oplus I_{n-k})(\bH_k\oplus \bM_{n-k})(V_1\oplus I_{n-k})\\
  &= (U_1V_1V_1^\ast \bH_k V_1)\oplus \bM_{n-k} =W_1\bH_k\oplus\bM_{n-k};\qquad W_1=U_1V_1.
\end{align*}
Unless $W_1=\pm I_k$ the set $W_1\bH_k$ differs from $\bH_k$ in which case the real dimension of $\Span_{\IR}(\cC\cup U\cC V)$ is larger than the real dimension of $\Span_{\IR}(\cC)$. Since this space contains the $T$-image of $\Span_{\IR}(\cC\cup U\cC V)$  we  contradict injectiveness of $T$. Thus, $U_1V_1=W_1=\pm I_k$. Consequently, $X_1=U_1D_1V_1=\pm U_1D_1U_1^\ast$ is either positive definite or negative definite. This shows that $T^{-1}(\cC \cup -\cC) \subseteq \cC \cup -\cC$.  Since $T$ also maps $\cC \cup -\cC$ into itself, $T(\cC \cup -\cC) = \cC \cup -\cC$.  Since $T$ is a homeomorphism, $T$ maps the disconnected interior $\cC^{\circ} \cup -\cC^{\circ}$ onto itself.  Since $T(A) = B\in\cC^\circ$, we must have $T(\cC^{\circ}) = \cC^{\circ}$.
\end{proof}

The next proposition proves that a bijective linear map on $\bM_k \oplus \bM_{n-k}$ satisfying $T({\cC}^\circ) = {\cC}^\circ$  and $T(\partial \cC) = \partial \cC$ maps rank one matrices of $\partial \cC$ into rank one matrices.
\begin{definition}
Given $X \in \partial \cC$, we define $$\Pert(X) = \{Z \in \partial \cC : X \pm tZ \in \cC \text{ for sufficiently small } t>0 \}.$$
\end{definition}

The following lemma concerns when the real dimension of $\Span_{\IR} (\Pert(X))$ equals one.

\begin{lemma}\label{prop:Pert}
Let $X \in \partial \cC$. Then $\dim_{\IR} \Span_{\IR} (\Pert(X))=1$ 
if and only if
	
\medskip
	(1) $X$ is psd and has rank one, or
	
	(2) $X = a (I_k \oplus P)$  for some $a > 0$ and unitary matrix $P$.
	\end{lemma}
\begin{proof}
Note $\Pert(0) = \{0\}$, so we may assume $X \ne 0$.  Next note that, for any unitary matrices $U_1 \in \bM_k$ and $V_2,W_2 \in \bM_{n-k}$, and any positive scalar $c$, the sets $\cC$ and $\partial \cC$ are invariant under the bijective map $\phi(X) = c(U_1 \oplus V_2) X (U_1^* \oplus W_2)$.  Consequently $Z \in \Pert(X)$ if and only if $\phi(Z) \in \Pert (\phi(X))$, and $\Pert(\phi(X)) = \phi(\Pert(X))$.  Thus we may assume that $X = \diag (s_1, \dots, s_n)$ with $1 = s_1 \geq \dots \geq s_n \geq 0$.

We prove necessity by showing the contrapositive.  Suppose $X$ has rank at least $2$, and $s_1 > s_n$.  Observe that $X \in \Pert(X)$, so it suffices to show $\Pert(X)$ contains some $Z \in \partial \cC$ that is not a real scalar multiple of $X$.  If $s_k > s_n$ then $X + dE_{nn} \in \Pert(X)$ for $d = (s_k - s_n)/2\in\IR$. If $s_k = s_n$ then $s_1 > s_k$ and $X + dE_{11} \in \Pert(X)$ for $d = (s_1 - s_k)/2\in\IR$.  Necessity follows.

To prove sufficiency, first suppose $X$ has rank one, so $X = E_{11}$.  Let $Z = Z_1 \oplus Z_2 \in \Pert(X)$.  Then $E_{11} \pm tZ_1$ is psd for some $t>0$, so $Z_1$ cannot have a nonzero diagonal entry except for the first one.  Since $Z \in \partial \cC$, $Z_1$ is a nonnegative scalar multiple of $E_{11}$ and $Z_2 = 0$.  Thus $\Pert(X)$ consists of nonnegative scalar multiples of $X$.

Next we may suppose $X = I_n$.  Let $0 \ne Z = Z_1 \oplus Z_2 \in \Pert(X)$. Since $Z \in \partial \cC$,  $s_k(Z_1) = s_1(Z_2)$. Also for sufficiently small $t > 0$, $I_n \pm tZ \in \cC$, that is, $I_k \pm tZ_1 \ge 0$ and $s_k(I_k \pm tZ_1) \ge s_1(I_{n-k} \pm tZ_2)$. Thus for sufficiently small $t > 0$, singular values of a psd matrix $I_k \pm tZ_1$ are its eigenvalues and
$$1 - t s_1(Z_1) = s_k(I_k - tZ_1) \ge s_1(I_{n-k} - tZ_2) = \|I_{n-k} - t Z_2\| \ge 1 - t\|Z_2\| = 1 - t s_1(Z_2) = 1 - t s_k(Z_1).$$
It follows that $s_1(Z_1) = s_k(Z_1)$ and $Z_1 = bI_k$ for some $b > 0$; by scaling $Z$ we may assume $b=1$.
Then for sufficiently small $t > 0$ and all unit vectors $v \in \IC^{n-k}$, $$1-t = s_k (I_k - t Z_1) \geq \|I_{n-k} - tZ_2 \| \geq |v^* (I_{n-k} - tZ_2) v| \geq \text{Re } v^* (I_{n-k} - tZ_2)v = 1 - t \text{Re } v^* Z_2 v,$$ so $1 \leq \text{Re } v^* Z_2 v$.  Thus the numerical range $W(Z_2)$ is contained in the closed half-plane $\text{Re } z \geq 1$ and the unit disk (because $\|Z_2 \| = 1$), so $W(Z_2) = \{1\}$, whence $Z_2 = I_{n-k}$ and $Z = I_n$.  Thus $\Pert(X)$ consists of nonnegative scalar multiples of $X$.  Sufficiency follows.
\end{proof}

With Lemma~\ref{prop:Pert} in mind we define two path-connected sets
\begin{align*}
  {\mathcal S}_1 &= \{X \in \partial \cC : \text{rank} \, X = 1\}= \{ aR \oplus 0_{n-k} : a > 0, R \in \bH_k \text{ is a rank-one projection}\}, \\
  {\mathcal S}_U&=\{a (I_k \oplus P) : a > 0, P \text{ is unitary} \}.
\end{align*}
Notice that, by continuity of singular values, their union is not path-connected.
\begin{proposition}\label{prop:permutes-CC}
 Suppose $1<k<n$ and $T$ is a bijective linear map on $\bM_k \oplus \bM_{n-k}$ satisfying $T(\cC) = \cC$ and $T(\partial \cC) = \partial \cC$.   Then $T$ maps ${\mathcal S}_1$ onto itself.
\end{proposition}

\begin{proof}
Being a linear bijection, $T$ maps $\Span_{\IR}(\Pert(X))$ onto $\Span_{\IR}(\Pert(T(X)))$ for each $X\in\partial \cC$. Hence, by Lemma~\ref{prop:Pert}, $T$ maps the subset ${\mathcal S}_1\cup {\mathcal S}_U\subseteq\partial\cC$ bijectively onto itself, so $T$ either maps the path connected components ${\mathcal S}_1$, $\mathcal{S}_U$ back onto themselves or swaps them.
The latter case would imply that the real-linear span of ${\mathcal S}_1$, that is,
$$\Span_{\IR}{\mathcal S}_1=\bH_k\oplus 0_{n-k}$$ would be mapped onto
$$\Span_{\IR}{\mathcal S}_U=\IR I_k\oplus \bM_{n-k}$$
(the equality holds  because each norm-one hermitian matrix $H$ is the midpoint of two unitaries $H\pm i \sqrt{I-H^2}$, so  each complex matrix is a real linear span of unitaries)
Thus $\dim_{\IR}\Span_{\IR}{\mathcal S}_1 = \dim_{\IR}\Span_{\IR}{\mathcal S}_U$ which gives
$k^2=1+2(n-k)^2$. Also, comparing the real dimensions of the $\IC$-linear spans of $\mathcal{S}_1$, $\mathcal{S}_U$, which $T$ maps in a like fashion,    we would get in addition $2k^2=2+2(n-k)^2$. Clearly  the two systems of equations are contradictory.
Thus $T$ maps $\mathcal{S}_1$ onto itself.
\end{proof}

With the above preparation, we can present the following. 

\begin{proof}[Proof of Theorem \ref{main2} for $1<k<n$] We first prove that if $T\colon \mathbb M_n\rightarrow \bM_n$ is a bijective linear map preserving parallel pairs with respect to $1<k<n$, then $T$ preserves rank one matrices also.

Let $R=cUE_{11} V \in\bM_n$ be a singular value decomposition of a rank-one matrix~$R$.
%
%
 Let $B=U(I_k\oplus 0_{n-k})V$ and $A=T^{-1}(B)$.  Then, by Corollary~\ref{cor:minimal}, $s_k(A)>s_{k+1}(A)$ so there exist unitary matrices $\hat{U},\hat{V}$ such that $\hat{U}A\hat{V}=\diag(s_1(A),\dots,s_n(A))=D_1\oplus D_2\in\cC^\circ$.

Introduce the parallelism preserver $\hat{T}(Z)=U^\ast T(\hat{U}^\ast Z\hat{V}^\ast)V^\ast$, which maps $D_1\oplus D_2\in\cC^\circ$ into  $U^\ast T(A)V^\ast=I_k\oplus 0_{n-k}\in\cC^\circ$. Then, $\hat{T}$ satisfies the hypotheses of Propositions~\ref{prop:phi-preservec-cone-CC} and then also of Proposition~\ref{prop:permutes-CC}. It follows that $\hat{T}$, hence also $\hat{T}^{-1}$, maps the set ${\mathcal S}_1\ni cE_{11}$ onto itself. Therefore, $\rk T^{-1}(R)=1$, and since $R$ was an arbitrary rank-one matrix, we see that $T^{-1}$  preserves the set of rank-one matrices.

Since $T^{-1}$ preserves rank-one matrices, by the  classical result on rank-one preservers (see, e.g.~\cite{Marcus-Molys} and~\cite{Westwick}) there exist invertible matrices $M,N \in \bM_n$ so that $T^{-1}(Y)=M^{-1}YN^{-1}$ or  $T^{-1}(Y)=M^{-1}Y^tN^{-1}$.  Finally, the result follows using Lemma~\ref{lem:rank-one=preservers}.
\end{proof}

\section{Results on real matrices}\label{section3}

In this section we characterize bijective linear maps $T\colon \bM_{n}(\IR)\rightarrow \bM_n(\IR)$ preserving parallel pairs with respect to the Ky-Fan $k$-norm for $1 < k \le n$.  As in the complex case we again obtain that all of them are scalar multiples of isometries.
Linear isometries  $T\colon  \bM_n(\IF) \rightarrow \bM_n(\IF)$ of the Ky-Fan $k$-norm $\|\cdot\|_{(k)}$ ($1 \le k \le n$) on $\bM_n(\IF)$ were
classified by the work of several authors; see \cite{Johnson-Laffey-Li,Grone-Marcus,Morita,Russo}. It turns out that, except for the isometries of the
Ky-Fan $2$-norm on $4$-by-$4$ real matrices,
 they  have the standard form. More precisely, the following holds:
\begin{itemize}
\item[(1)] $T$ has the form $X\mapsto UXV$ or $X\mapsto UX^tV$
for some  $U,V\in\bM_n(\IF)$ satisfying
$U^*U = V^*V = I_n$, i.e., $U, V$ are unitary if $\IF = \IC$, and
$U, V$ are orthogonal if $\IF = \IR$, or
\item[(2)] $(\bM_n(\IF), k) = (\bM_4(\IR), 2)$ and $T$ has the form
$$X\mapsto UX^+V\quad \hbox{ or }\quad X\mapsto \IL(UX^+V) \quad
\hbox{ or }\quad X\mapsto -E\IL(EUX^+V)$$
where $X^+$ denotes either the identity or the transpose map, and
$E=\diag(1,-1,-1,-1)$, $U,V\in \bM_4(\IR)$ are orthogonal, and
 $\IL\colon\bM_4(\IR)\to\bM_4(\IR)$ is defined by
$$\IL(X)=\tfrac{1}{2}\bigl(X+B_1X C_1+B_2XC_2+B_3XC_3\bigr),$$
where
$$B_1=\left(
\begin{array}{cc}
 1 & 0 \\
 0 & 1 \\
\end{array}
\right)\otimes \left(
\begin{array}{cc}
 0 & -1 \\
 1 & 0 \\
\end{array}
\right),\quad
C_1=\left(
\begin{array}{cc}
 1 & 0 \\
 0 & -1 \\
\end{array}
\right)\otimes \left(
\begin{array}{cc}
 0 & 1 \\
 -1 & 0 \\
\end{array}
\right),
$$
$$
B_2=\left(
\begin{array}{cc}
 0 & 1 \\
 -1 & 0 \\
\end{array}
\right)\otimes \left(
\begin{array}{cc}
 -1 & 0 \\
 0 & 1 \\
\end{array}
\right),\quad
 C_2=\left(
\begin{array}{cc}
 0 & 1 \\
 -1 & 0 \\
\end{array}
\right)\otimes \left(
\begin{array}{cc}
 1 & 0 \\
 0 & 1 \\
\end{array}
\right),
$$
$$
B_3=\left(
\begin{array}{cc}
 0 & -1 \\
 1 & 0 \\
\end{array}
\right)\otimes \left(
\begin{array}{cc}
 0 & 1 \\
 1 & 0 \\
\end{array}
\right), \quad  C_3=\left(
\begin{array}{cc}
 0 & 1 \\
 1 & 0 \\
\end{array}
\right)\otimes \left(
\begin{array}{cc}
 0 & 1 \\
 -1 & 0 \\
\end{array}
\right).$$
\end{itemize}
Let $\mathrm{SO}(n)\subseteq\bM_n(\IR)$ be the group of orthogonal matrices with determinant $1$ and let $\bS_n\subseteq\bM_n(\IR)$ be the set of $n$-by-$n$ symmetric matrices.
As shown in~\cite{Johnson-Laffey-Li}, the map $\IL$ satisfies the following.
\begin{itemize}
 \item $\IL$ is an involution.
 \item $\IL$ maps rank-one matrices of norm-two in $\bS_2\oplus 0_2$ onto $I_2\oplus \mathrm{SO}(2)$.
 \item $\IL$ fixes the set of orthogonal matrices with negative determinant.
\end{itemize}

The above facts will be used in the proof of our main theorem of this section.

\begin{theorem}\label{main}  Let $\|\cdot\|_{(k)}$ be the
Ky-Fan $k$-norm on $\bM_n(\IR)$ for $1 < k \le n$. Suppose $T\colon \bM_n(\IR) \rightarrow \bM_n(\IR)$ is a bijective linear map which preserves
parallel pairs with respect to $\|\cdot\|_{(k)}$. Then $T$ is a scalar multiple of a linear isometry of $(\bM_n(\IR),\|\cdot\|_{(k)})$. More precisely, there exists $\gamma > 0$ and orthogonal matrices $U, V\in \bM_n(\IR)$ such that $T$ has the form
$$X \mapsto \gamma UXV \qquad \hbox{ or } \qquad X \mapsto \gamma UX^tV,$$
except when $(k, n) = (2,4)$, in which case $T$ either  takes one of the above two forms~or
$$X\mapsto\gamma P\IL(UXV)\qquad \hbox{ or } \qquad X \mapsto \gamma P\IL(UX^tV)$$
 for some scalar  $\gamma > 0$ and orthogonal $U, V,P \in \bM_4(\IR)$.
\end{theorem}

\medskip

\begin{corollary}
 Assume a linear bijection $T\colon\bM_n(\IR)\to\bM_n(\IR)$ is such that $\|A+B\|_{(k)}=\|A\|_{(k)}+\|B\|_{(k)}$ implies $\|T(A)+T(B)\|_{(k)}=\|T(A)\|_{(k)}+\|T(B)\|_{(k)}$,
 where $\|\cdot\|_{(k)}$ is a Ky-Fan $k$-norm. Then, $T$ is a scalar multiple of a linear isometry of $(\bM_n(\IR),\|\cdot\|_{(k)})$ and hence takes one of the forms from Theorem~\ref{main}.
\end{corollary}

Our presentation for the proof of Theorem \ref{main} is similar to that of Theorem \ref{main2} in Section \ref{section2}. Instead of $\bH_k\subseteq\bM_k(\IC)$, we will be dealing
with $\bS_k\subseteq\bM_k(\IR)$,  the set of all symmetric real $k$-by-$k$ matrices. And instead of unitary matrices, we will be dealing with orthogonal matrices.
Also, the complex unit  in this case is real and therefore equal to $1$ or $-1$.
The results with proofs follow along the same lines as in Section \ref{section2}, except when $(k,n)=(2,4)$, where we have additional isometries.

We note that the proof for the trace-norm case follows exactly along the same lines as the complex case in Section \ref{section2}.
However, for Ky-Fan $k$-norms with $2\le k\le n-1$,
there are peculiarities in the real case.  One obvious difference from the complex case is the presence of special isometries when $(n,k) = (2,4)$.  Another more subtle difference is that the real span of psd matrices is $\bS_k$ in the real case, instead of $\bH_k$ in the complex case, with different (real) dimensions ($k(k+1)/2$ versus $k^2$); another contrast is that the complex span of psd matrices is $\bM_k$.


With this in mind, the analogous statement of Lemma \ref{Lem:dimP(A)} is the following (its proof proceeds along  the same lines as for Lemma~\ref{Lem:dimP(A)} and will not be repeated here). Note that $X^* = X^t$
if $X$ is a real matrix.
\begin{lemma}\label{Lem1:dimP(A)} Let $A \in \mathbb \bM_n(\IR)$ and let
$U, V \in \bM_n(\IR)$ be orthogonal matrices such that $U^*AV = \diag(s_1(A), \dots, s_n(A))$. Then, the following holds.
\begin{itemize}
\item[{\rm (a)}] Suppose $s_k(A)> s_{k+1}(A)$. Then
$B \in P(A)$ if and only if $U^*BV = B_1 \oplus B_2$ such that $\mu B_1 \in \bM_k(\IR)$ is  psd for some $\mu\in \{-1,1\}$ and $|\tr B_1| = \|B\|_{(k)}$. Moreover,
$$ \Span P(A) =  U(\bS_k \oplus \bM_{n-k}(\IR))V^*.$$
Consequently, $\Span P(A)$ has dimension $k(k+1)/2 + (n-k)^2$.
\item[{\rm (b)}]  Suppose $s_k(A) = s_{k+1}(A)$. Let $p$ and $q$ be the smallest integer
and the largest integer such that $s_{p+1}(A) = s_k(A) = s_q(A)$.
\begin{itemize}
\item If $q \ne n$ or $A$ is invertible, then $B \in P(A)$ if and only if there is an orthogonal matrix $W$ of the form $W = I_p \oplus W_1 \oplus I_{n-q} \in \bM_p(\IR) \oplus \bM_{q-p}(\IR) \oplus \bM_{n-q}(\IR)$ such that $B = U W (B_1 \oplus B_2) W^* V^*$, $\mu B_1\in \bM_k(\IR)$ is  psd for some
$\mu \in \{-1,1\}$,
and $\|B_1\|_{(k)} = \|B\|_{(k)}$.
	
\item If $q=n$ and $A$ is singular, then $B \in P(A)$ if and only if there exist orthogonal matrices $W_2, W_3 \in \bM_{n-p}(\IR)$ such that $B = U (I_p \oplus W_2) (B_1 \oplus B_2) (I_p \oplus W_3) V^*$, $\mu B_1\in \bM_k(\IR)$ is  psd for some $\mu \in \{-1,1\}$,
and $\|B_1\|_{(k)} = \|B\|_{(k)}$.
\end{itemize}
 Consequently, $\Span P(A)$ has dimension at least $k(k+1)/2+k(q-k)+(n-k)^2$.
\end{itemize}
\end{lemma}

Corollary \ref{cor:SASB} and Corollary \ref{cor:minimal} also hold for $\bM_n(\mathbb R)$. Note that Item (ii) of Corollary \ref{cor:SASB} is not possible for matrices over real fields because
for Item (ii) we would have
$\Span P(B)=\bM_k(\IR)\oplus (X_2\bS_{k}Y_2^\ast) \neq \Span P(A) =\bS_k \oplus \bM_k$ (here, $X_2,Y_2$ are the second diagonal blocks of orthogonal matrices $X,Y$). We define $\cC$ in exactly the same way as in Definition \ref{defn1}. We just note that its affine hull is $\Span_{\IR}\cC=\bS_k\oplus\bM_{n-k}(\IR)$ and it satisfies all other properties mentioned in Proposition~\ref{prop1}.

The proof of Proposition \ref{prop:phi-preservec-cone-CC} requires a separate proof for real fields.  We restate the proposition for the real case here for convenience. 

\begin{proposition}\label{prop1:phi-preservec-cone-CC}
 Let  $T\colon\bM_n(\IR)\to\bM_n(\IR)$ be  a bijective linear map which preserves parallel pairs relative to a Ky-Fan $k$-norm for $1< k< n$. Assume there exist $A,B\in{\cC}^\circ$ such that $T(A)=B$. Then $$T({\cC}^\circ) = {\cC}^\circ\quad\hbox{  and }\quad T(\partial \cC) = \partial \cC.$$ Moreover, $T$ maps  $\Span P(B)=\Span P(A)$ bijectively onto itself.
\end{proposition}
The reason for a separate proof, roughly speaking, is because in the complex case, $P(A)\setminus\{0\}=\IT(\cC\setminus\{0\})$ (with $\IT:=\{\mu\in\IC:|\mu|=1\}$) is a connected set whenever $A\in{\cC}^\circ$  while in the real case the blunted cones $\cC\setminus\{0\}$ and $-\cC\setminus\{0\}$ are the two connected components of $P(A)\setminus\{0\}$. In the subsequent proofs, we will typically write $\bM_n$ instead of $\bM_n(\IR)$ for notational simplicity.

\begin{proof}
The last claim follows the proof of Proposition \ref{prop:phi-preservec-cone-CC}
except that now, Lemma~\ref{Lem1:dimP(A)} gives $\Span P(A)=\Span P(B)=\bS_k\oplus\bM_{n-k}$ is invariant for~$T$.
We now show $T(\cC^{\circ}) \subseteq \cC^{\circ}$.
Every matrix $X \in \cC$ is parallel to $A$. So, $T(X) = Y \in \bS_k \oplus \bM_{n-k}$ is parallel to  $B$.  So, $Y = Y_1 \oplus Y_2$ with $\mu Y_1$ positive semi-definite for some $\mu\in\{-1,1\}$ and $s_k(Y_1) \ge s_1(Y_2)$. So, $T(\cC) \subseteq \cC \cup -\cC$. Now, $\cC$ and $-\cC$ are connected cones with intersection equal to $\{0\}$. We may replace $T$ by $-T$ if needed, and assume that $T(\cC) \subseteq \cC$. Since $T\colon \bS_k\oplus \bM_{n-k} \rightarrow \bS_k \oplus \bM_{n-k}$ is bijective linear it will send interior points to interior points.

We now claim that $T^{-1}(\cC^{\circ}) \subseteq \cC^{\circ}$.
To see this,
let $Y \in \cC^{\circ}$, so
$Y = Y_1 \oplus Y_2\in \bS_k \oplus \bM_{n-k}$ with 
$Y_1$ 
positive definite and $s_k(Y_1) > s_1(Y_2).$  If
$T(X) = Y$, then, by Corollary~\ref{cor:minimal}, $T^{-1}(\Span P(Y))=\Span P(X)$, $s_k(X) > s_{k+1}(X),$
and $X = X_1 \oplus X_2\in\bS_k\oplus\bM_{n-k}$.
Moreover, since the relative interior of $\cC$ is again a convex cone, for all $t > 0,$ \begin{equation}\label{Yt}
\tilde{Y}_t = (B_1\oplus B_2) + t(Y_1 \oplus Y_2) \hbox{ satisfies } s_k(\tilde{Y}_t) > s_{k+1}(\tilde{Y}_t).
\end{equation}
By Corollary~\ref{cor:minimal}, $\tilde{X}_t = (A_1 + tX_1) \oplus (A_2 + t X_2)$ satisfies $s_k(\tilde{X}_t) > s_{k+1}(\tilde{X}_t)$. Suppose, by way of contradiction, that $s_1(X_2)$ is one of the $k$ largest singular values of $X$. Then $s_{1}(X_2) > s_k(X_1)$. For large $t> 0$, $s_1(A_2 \oplus t X_2)$ is one of the $k$ largest singular values of $A + tX$, but for very small $t>0$, $s_k(A_1 + tX_1) > s_1(A_2+tX_2)$. So, there will be a $t_0 > 0$ such that $s_1(A_1+t_0 X_1) \ge \dots \ge s_k(A_1+t_0 X_1) = s_1(A_2 + t_0 X_2)$,  contradicting $s_k(\tilde{X}_t) > s_{k+1}(\tilde{X}_t)$.  Thus $s_k(X_1) > s_1(X_2)$.

Observe that $\Span P(Y)=\Span P(B)=\bS_k\oplus\bM_{n-k}$ and this set is invariant for $T^{-1}$. Then also $ \Span P(X)=\bS_k\oplus\bM_{n-k}$. Now choose an orthogonal matrix $O_1\in\bM_k$ such that $D_1=O_1X_1O_1^t$ is diagonal. Let $O=O_1\oplus I_{n-k}$ and let $D_1^{(s)}=|D_1|^{-1}D_1$ {(so $D_1^{(s)}$ is an orthogonal diagonal matrix)}. Then, $\Span(P(OXO^t))=(D_1^{(s)}\oplus I_{n-k})\Span(P(|D_1|\oplus X_2))=(D_1^{(s)} \bS_k)\oplus \bM_{n-k}$, which equals $O\Span P(X)O^t = \bS_k\oplus \bM_{n-k}$ if and only if $D_1^{(s)}=\pm I_k$. Hence $X_1$ is  either positive or negative definite.

Thus $T^{-1}$ maps the connected blunted cone $\cC\setminus\{0\}$ into the disconnected set $(\cC\setminus\{0\})\cup ((-\cC)\setminus\{0\})$,
whence $T^{-1}(\cC^{\circ}) \subseteq \cC^{\circ}$.
This shows that $T(\cC^\circ) = \cC^\circ$. Being a homeomorphism, $T$ will then map $\partial \cC$, which coincides with  the boundary of $\cC^\circ$, onto the boundary of $T(\cC^\circ)=\cC^\circ$, as claimed.
\end{proof}

As before, with  Lemma~\ref{prop:Pert} in mind, we define three sets:
\begin{align*}
{\mathcal S}_1 &= \{X \in \partial \cC : \text{rank} \, X = 1\} = \{ aR \oplus 0_{n-k} : a > 0, R \in \bS_k \text{ is a rank-one projection}\}, \\
{\mathcal S}_+ &= \{a (I_k \oplus P) : a > 0, P \text{ is orthogonal with } \det P = 1\},\\
{\mathcal S}_- &= \{a (I_k \oplus P) : a > 0, P \text{ is orthogonal with } \det P = -1\}.
\end{align*}
Observe that ${\mathcal S}_1, {\mathcal S}_+, {\mathcal S}_-$ are mutually disjoint. Moreover, they are path-connected. For ${\mathcal S}_1$ this holds because if $X \in  {\mathcal S}_1 \setminus \IR E_{11}$, then $X = xx^\ast$  for some vector $x \in \IR^k \oplus 0_{n-k}$, linearly independent of $e_1$, and hence $\lambda\mapsto (\lambda e_1+(1-\lambda) x)(\lambda e_1+(1-\lambda) x)^\ast$ is a path in ${\mathcal S}_1$ connecting $X$ to $E_{11}$. For ${\mathcal S}_+$ its path-connectedness follows by using a parametrization of the special orthogonal group via $e^X$  where the parameter $X=-X^t$  runs over (the path connected set of) all skew-symmetric matrices, while
\begin{equation}\label{eq:S-}
  {\mathcal S}_-=(I_k\oplus\diag(1,\dots1,-1)){\mathcal S}_+
\end{equation}
is clearly also path-connected. Moreover, by using the determinant we see that there can be no path between $\cS_-$ and $\cS_+$ and by using the continuity of singular values we see that neither $\cS_1\cup \cS_-$ nor  $\cS_1\cup\cS_+$ are  path connected.

We will further need that \begin{equation}\label{eq:(n-k)>2}
 \Span {\mathcal S}_1 = \bS_k \oplus 0_{n-k}\quad \hbox{and, for } n - k \geq 3,  \quad \Span {\mathcal S}_+ = \Span {\mathcal S}_- = \IR I_k\oplus\bM_{n-k}(\IR).
\end{equation}
(To see the last two equalities, note that, for any orthonormal set  $y_1,\dots,y_{n-k}\in\IR^{n-k}$, one has $I_k\oplus Y_{\pm}\in {\mathcal S}_+$  where $Y_{\pm} = \pm(y_1 y_1^t + y_2 y_2^t ) +\sum_{j=3}^{n-k} y_j y_j^t$; thus $0_k \oplus R \in \Span {\mathcal S}_+$ for every rank-two projection $R$. One then deduces that $0_k\oplus R \in  \Span {\mathcal S}_+$ for every rank-one projection $R$, giving $\Span {\mathcal S}_+ = \IR I_k \oplus \bM_{n-k}(\IR)$. One applies \eqref{eq:S-} to get the first equality.) In contrast, if $n -k = 2$, then
 \begin{equation}\label{eq:n-k=2}
  \Span {\mathcal S}_+=\IR I_k\oplus \Bigl(\IR\left(\begin{smallmatrix}
             1 & 0 \\
             0 & 1
              \end{smallmatrix}\right)+\IR  \left(\begin{smallmatrix}
             0 & 1 \\
             -1 & 0
              \end{smallmatrix}\right)\Bigr)\quad\hbox{and}\quad \Span {\mathcal S}_{-}=\IR I_k\oplus \Bigl(\IR\left(\begin{smallmatrix}
             1 & 0 \\
             0 & -1
              \end{smallmatrix}\right)+\IR\left(\begin{smallmatrix}
             0 & 1 \\
             1 & 0
              \end{smallmatrix}\right)\Bigr)
 \end{equation}
while  if $n-k = 1$, then
\begin{equation}\label{eq:n-k=1}
\Span {\mathcal S}_+ = \IR I_n \quad\hbox{and }\quad \Span {\mathcal S}_- = \IR\,\diag(1,\dots,1,-1).
\end{equation}
\medskip

The next result is the  counterpart to Proposition~\ref{prop:permutes-CC} over the  field of real numbers.
\begin{proposition}\label{prop:permutes}
 Suppose $1<k<n$ and $T$ is a bijective linear map on $\bS_k \oplus \bM_{n-k}(\IR)$ satisfying $T(\cC) = \cC$ and $T(\partial \cC) = \partial \cC$.   Then $T$ permutes the sets ${\mathcal S}_1$, ${\mathcal S}_+$, ${\mathcal S}_-$ among themselves; actually, $T$
fixes ${\mathcal S}_1$ except possibly when $(k, n) = (2, 4)$.
\end{proposition}
\begin{proof}

Being a linear bijection, $T$ maps $\Span(\Pert(X))$ onto $\Span(\Pert(T(X)))$ for each $X\in\partial \cC$.
Hence, by Proposition~\ref{prop:Pert}, $T$ maps the subset ${\mathcal S}_1\cup {\mathcal S}_+\cup {\mathcal S}_-$ bijectively onto itself. Since ${\mathcal S}_1$, ${\mathcal S}_+$ and ${\mathcal S}_-$
are all connected and disjoint, their images under a homeomorphic map $T$ are also connected
and disjoint. Then $T$ merely permutes the sets ${\mathcal S}_1$, ${\mathcal S}_+$, and ${\mathcal S}_-$ among themselves. Consequently,
$T$ also permutes the sets $(\Span {\mathcal S}_1)$, $(\Span {\mathcal S}_+)$, and $(\Span {\mathcal S}_-)$ among themselves with the same
permutation as for ${\mathcal S}_1$, ${\mathcal S}_+$, ${\mathcal S}_-$.



Now, if $n-k\ge 3$, then $(\Span {\mathcal S}_1)=\bS_k\oplus 0_{n-k}$, while $\Span {\mathcal S}_+ = \Span {\mathcal S}_-= \IR I_k\oplus \bM_{n-k}(\IR) $. Since $T$ permutes $S_1,S_-,S_+$ among themselves and fixes the cone $\cC$ we see that,  unless  $T$ fixes ${\mathcal S}_1$, it would map the cone  $$(\Span {\mathcal S}_1)\cap\cC=\mathrm{psd}_k\oplus 0_k$$ bijectively onto the cone
$$(\Span {\mathcal S}_+)\cap \cC=(\Span {\mathcal S}_{-})\cap \cC=\{\lambda I_k\oplus X: s_1(X)\le \lambda \}.$$
 Then also the extreme rays of the cone $(\Span {\mathcal S}_1)\cap\cC$  would be mapped onto the extreme rays of the cone $(\Span {\mathcal S}_+)\cap\cC$. 
 It is well-known~\cite[Theorem 3.8]{HW} that the extreme rays of the first cone are of the form $\IR_+ xx^*$, i.e., the collection of all extreme rays of the first cone is $\cS_1$.
  On the other hand, the collection of all extreme rays of the second cone contains ${\mathcal S}_-\cup{\mathcal S}_+$.
(To see this, let $P \in \bM_{n-k}(\IR)$ be orthogonal.  Suppose $I_k \oplus P = \frac{1}{2}(\alpha I_k \oplus X + \beta I_k \oplus Y)$ is an average of two elements from the second cone.  Then $\alpha + \beta = 2$ and $P = \frac{\alpha}{\alpha + \beta} (X/\alpha) + \frac{\beta}{\alpha + \beta}(Y/\beta)$ is a convex combination of two matrices $X/\alpha$, $Y/\beta$ in the unit ball.  Since orthogonal matrices are extreme points of the unit ball~\cite[Theorem~3]{Grone-Marcus}, $X/\alpha = Y/\beta = P$, so $\alpha I_k \oplus X = \alpha (I_k \oplus P)$ and $\beta I_k \oplus Y = \beta (I_k \oplus P)$.  Thus each element of ${\mathcal S}_-\cup{\mathcal S}_+$ lies in an extreme ray of the second cone, as claimed.)
  But then $T$ would map the collection of all extreme rays of the first cone,  ${\mathcal S}_1$, onto a set which contains ${\mathcal S}_-\cup {\mathcal S}_+$, a contradiction to the fact that $T$ permutes these three sets.
This implies that $T({\mathcal S}_1) = {\mathcal S}_1$ if $n-k\ge 3$.

If $n-k = 2$ and $k\ge3$, then $\dim(\Span {\mathcal S}_1) = \frac{k(k + 1)}{2} > 3 = \dim(\Span {\mathcal S}_+) =\dim(\Span {\mathcal S}_-)$
(see \eqref{eq:n-k=2}), so again $T({\mathcal S}_1) = {\mathcal S}_1$. If $n-k=1$ we argue in the same manner
(see \eqref{eq:n-k=1}).

The only remaining possibility is $(k, n) = (2, 4)$.
\end{proof}

\begin{proof}[Proof of Theorem \ref{main}] We only need to  prove that there exists a linear isometry $\Phi\colon (\bM_n,\|\cdot\|_{(k)}) \to (\bM_n,\|\cdot\|_{(k)})$
such that $(T\circ\Phi)^{-1}$ preserves the set of rank-one matrices. Once this is done, $(T\circ\Phi)$ also preserves rank-one matrices and   takes the standard form $X\mapsto MXN$ or $X\mapsto MX^t N$ (see~\cite{Westwick}) and the rest follows from Lemma~\ref{lem:rank-one=preservers}.

To this end, let $R_1 = xy^\ast$ and $R_2 = uv^\ast$ be two arbitrary rank-one matrices in $\bM_n$ with a unit Ky-Fan
norm. 
We can assume that
the vectors $x$, $y$, $u$, $v$ are all normalized. Decompose
$$u = cx + sx_2\quad\hbox{ and }\quad v = \hat{c}y + \hat{s}y_2; \qquad |c|^2 + |s|^2 = |\hat{c}|^2 + |\hat{s}|^2 = 1,\quad s, \hat{s}\ge0$$
where $x_2,y_2\in\IR^n$ are unit vectors orthogonal to $x$ and $y$, respectively. Choose a point
$(\alpha, \beta)$ on the unit circle of $\IR^2$ with $(c-\hat{c})\alpha + (s-\hat{s})\beta = 0$ and form a norm-one matrix
$$R = (\alpha x +\beta x_2)(\alpha y + \beta y_2)^\ast.$$
Then there exist orthogonal matrices $U$, $V$ such that $Ux = V y = e_1$ and $Ux_2 = V y_2 = e_2$, the first
two standard basis vectors of $\IR^n$, and we have
\begin{equation}\label{eq:R1-R}
  UR_1V^\ast = E_{11}\quad\hbox{  and }\quad URV^\ast = (\alpha e_1 +\beta e_2)(\alpha e_1 + \beta e_2)^\ast\in \bS_2\oplus 0_{n-2}\subseteq \bS_k \oplus  0_{n-k}.
\end{equation}
Moreover, by the choice of $(\alpha, \beta )$ we further have $u^\ast (\alpha x + \beta x_2) = \alpha c + \beta s = v^\ast (\alpha y + \beta y_2)$ so there
exist orthogonal $U_2, V_2\in\bM_n$ with
\begin{equation}\label{eq:R-R2}
U_2RV_2^\ast = E_{11}\quad\hbox{ and }\quad U_2R_2V_2^\ast \in \bS_2(\IR) \oplus 0_{n-2}.
\end{equation}
Consider first the pair $(R_1,R)$ in \eqref{eq:R1-R}. Extend $R_1$ to a matrix $Y = U^\ast (E_{11} + \dots+ E_{kk})V =
U^{\ast}(I_k \oplus  0_{n-k})V$ and let $A = T^{-1}(Y )$. Since $s_k(Y ) > s_{k+1}(Y )$, Corollary~\ref{cor:minimal} implies that $s_k(A) > s_{k+1}(A)$ so there exist orthogonal $\hat{U},\hat{V}$ such that $A =\hat{U}(A_1 \oplus  A_2)\hat{V}\in \bS_k \oplus \bM_{n-k}$ with $A_1$ positive definite and $s_k(A_1) > s_1(A_2)$. By replacing $T$ with a
map
$$\hat{T}\colon X\mapsto UT(\hat{U} X \hat{V})V^\ast $$
we achieve that $\hat{T}(A_1\oplus A_2) = I_k \oplus 0_{n-k}$. Then, by Corollary~\ref{cor:minimal},
$\hat{T}$ maps $\Span P(A_1 \oplus  A_2) = \bS_k \oplus  \bM_{n-k}$ onto $\Span P(I_k \oplus  0_{n-k}) = \bS_k \oplus  \bM_{n-k}$, so the conclusions
of Propositions~\ref{prop:phi-preservec-cone-CC} and \ref{prop:permutes} are valid. We next consider two cases.

Assume $(k, n) \neq (2, 4)$. Then, by Proposition \ref{prop:permutes}, $\hat{T}$, hence also $\hat{T}^{-1}$, fixes the set ${\mathcal S}_1$ which by
definition consists of all psd rank-one matrices in $\bS_k \oplus 0_{n-k}$. It follows that $\hat{T}^{-1}(E_{11}) = \hat{U}^\ast T^{-1}(R_1)\hat{V}^\ast$ is
of rank-one. Since $R_1$ was arbitrary, the linear bijection $T^{-1}$ preserves rank-one matrices and the
claim follows with the isometry $\Phi(X) = X$.

Assume $(k, n) = (2, 4)$. Then either $\hat{T}^{-1}$ fixes ${\mathcal S}_1$ or else maps it onto ${\mathcal S}_+$ or onto ${\mathcal S}_-$. In the last two
cases, $\hat{T}^{-1}$ maps every rank-one in $\bS_2\oplus 0_2$ into an invertible matrix of the form $\IR(I_2 \oplus \mathrm{O}(2))$,
where $\mathrm{O}(2) \subseteq \bM_2$ denotes the group of $2$-by-$2$ orthogonal matrices. By \eqref{eq:R1-R} and the definition of $\hat{T}$ it
follows that $T^{-1}$ either maps $R_1$ and $R$ simultaneously into rank-one matrices or else simultaneously
into invertible matrices. We now repeat the above procedure on the pair $(R,R_2)$ using \eqref{eq:R-R2} instead
of \eqref{eq:R1-R} to see that either $T^{-1}(R_1)$ and $T^{-1}(R_2)$ are both rank-one matrices, or else they are both
invertible ones. Then, by the arbitrariness of $R_1$ and $R_2$,  either $T^{-1}$ preserves the set of rank-one matrices
or else it maps it into the set of invertible ones and the same applies to $\hat{T}^{-1}$. In the former case
we are done by using the isometry $\Phi(X) = X$. In the latter case, $\hat{T}^{-1}$ maps the set ${\mathcal S}_1$ onto ${\mathcal S}_+$ or onto
${\mathcal S}_-$. Composing $\hat{T}$ with an involutive isometry $\IL$ or  isometry $\bL\colon X\mapsto \diag(1, 1, 1,-1)\IL(X)$,
respectively, we get that $(\hat{T}\circ\IL)^{-1}$ or $(\hat{T}\circ\bL)^{-1}$ fixes ${\mathcal S}_1$, so preserves rank-one. We finally compose $T$ with the isometry $\Phi(X)=\hat{U}\IL(X)\hat{V}$ or the isometry $\Phi(X)=\hat{U}\bL(X)\hat{V}$; its inverse, $(T\circ\Phi)^{-1}$ will map
rank-one matrices to rank-one matrices, as claimed.
\end{proof}

\medskip

\section{Final remarks and future study}

\begin{enumerate}

\item In all our results on Ky-Fan norms, if a bijective linear  $T$  preserves parallel pairs, then its inverse does so  also. However, this is not true for a general norm. For example, consider on $\IR^2$ the truncated Euclidean norm
       $$\|(a,b)\|_{\mathrm{trE}}:=\left\| \Bigl(\| (a,b)\|_2 \,,\,\bigl\|
   (\sqrt{2} a,b)\bigr\|
   _{\infty }\Bigr)\right\| _{\infty
   }=\begin{cases}
 \sqrt{a^2+b^2}; & | b| >| a|  \\
 \sqrt{2} | a|;  & \text{otherwise}
\end{cases}$$
One can show that $x\|y$ if and only if $x,y$ are linearly dependent or else they both belong to $\sigma\cup(-\sigma)$, where $\sigma$ is the pointed convex cone spanned  by a vertical line segment through $(1,1)$ and $(1,-1)$. Then any linear map $T$ which maps $\sigma$ into itself is a parallelism preserver with respect to $\|\cdot\|_{\mathrm{trE}}$. But if $T(\sigma) $ is properly contained in $\sigma$, it will not preserve parallel pairs in both directions. A concrete example is $T=\left(
\begin{smallmatrix}
 1 & 0 \\
 \frac{1}{12} & \frac{5}{12}
   \\
\end{smallmatrix}
\right)$.  In another venue, $T=\tfrac{1}{2}\left(
\begin{smallmatrix}
 3 & 1 \\
 1 & 3 \\
\end{smallmatrix}
\right)$ preserves parallel pairs of $\|\cdot\|_{\mathrm{trE}}$ in both directions, but it is not a scalar multiple of an isometry of $\|\cdot\|_{\mathrm{trE}}$.
\item  Observe that without assuming linearity or bijectivity  there are more pathological examples.  A trivial one is $T(A) =  f(A) A$
where  $f\colon\bM_n\to\IC$ can be any function (this  is  bijective when $f$ maps every line, spanned by any fixed nonzero matrix $A$, bijectively onto $\IC\setminus\{0\}$); a slightly more subtle one is $T(A) = F_1(A) \oplus 0$, where $F_1\colon \bM_n \to \mathrm{psd}_k$ is an arbitrary map. Finally,
the proofs are much easier if one assumes $T$ preserves
parallel pairs in both directions.
\item  By our results,  the  bijective linear preservers of parallel pairs and  
of matrix pairs $(A,B )$
satisfying $\|A+B\|_{(k)} = \|A\|_{(k)} + \|B\|_{(k)}$  are both positive multiples of isometries.
However, this is not always the case for other norms. For example,
in \cite{LTWW}, it was shown that the same conclusion holds
on $\bM_n$ equipped with the  spectral norm (i.e., Ky-Fan $1$-norm) if $n \ge 3$.
However, when $n = 2$, then in the complex case there are additional maps preserving parallel pairs
but no additional maps preserving matrix pairs $(A, B )$ satisfying $\|A+B\| = \|A\| + \|B\|$, while in the real case there are additional maps preserving parallel pairs that
will also preserve matrix pairs $(A, B)$ satisfying $\|A+B\| = \|A\| + \|B\|$.

\item One may consider bijective  linear maps on
complex or real rectangular matrices $\bM_{m,n}$
preserving parallel pairs with respect to the Ky-Fan $k$-norm.
They should be scalar multiples of isometries.


\item More generally, one may consider general unitarily invariant norms on $\bM_{m,n}$ that are
not strictly convex. For example, the $(c,p)$-norm
$\|A\| = (\sum_{j=1}^k c_j s_j(A)^p)^{1/p}$.

\item It would be interesting to extend our results to $B(H)$, the set of bounded linear operators
acting on a real or complex Hilbert space $H$, under the Ky-Fan $k$-norm
$$||A||_{(k)} =
\sup \{ ||X^*AY||_{(k)}:  X^*X = Y^*Y = I_k\}.$$
 \item One may also consider extending the results to infinite dimensional operators in other spaces, such as norm ideals of compact operators or  standard operator algebras.

\item Another direction is to extend the results to other matrix or operator algebras.
For instance, for finite dimensional irreducible algebras, the problem reduces
to the study of real or complex square matrices. For a real matrix algebra one
needs to consider the algebra of complex matrices or quaternion matrices over reals.
One may also consider the problem on the algebra of triangular matrices.
The infinite dimensional extension to nested algebras would be of interest too.

\end{enumerate}
\bigskip
\noindent
{\bf \Large Acknowledgment}

The research of Li was
supported by the Simons Foundation Grant 851334. This  research was supported in part by the Slovenian Research Agency (research program P1-0285, research project N1-0210, and bilateral project BI-US-22-24-129).

\bigskip\bigskip

\noindent
(Kuzma)
Department of Mathematics, University of Primorska, Slovenia; Institute of Mathematics,
Physics,
and Mechanics, Slovenia. E-mail: bojan.kuzma@famnit.upr.si

\medskip\noindent
(Li) Department of Mathematics, College of William \& Mary, Williamsburg, VA 23187, USA.
ckli@math.wm.edu

\medskip\noindent
(Poon)
Department of Mathematics, Embry-Riddle Aeronautical University, Prescott AZ 86301, USA.
E-mail: poon3de@erau.edu

\medskip\noindent
(Singla)
Department of Mathematics, University of Primorska, Slovenia. E-mail: ss774@snu.edu.in


\begin{thebibliography}{WW}





\bibitem{D} J. Dieudonn\'{e}, Sur une g\'{e}n\'{e}ralisation du groupe orthogonal
\'{a} quatre variables, Arch. Math, 1 (1949), 282-287.

\bibitem{Grone-Marcus} R. Grone, M. Marcus, Isometries of matrix algebras, J. Algebra 47 (1977), no. 1, 180--189.

\bibitem{HW} R. Hill, S. Waters, On the cone of positive semidefinite matrices, Linear Algebra Appl. 90 (1987), 81-88.

\bibitem{Johnson-Laffey-Li}  C.R. Johnson, T. Laffey, and  C.-K. Li, Linear transformations on $M_n(\IR)$ that preserve the Ky
Fan $k$-norm and a remarkable special case when $(n,k) = (4, 2)$. Linear and Multilinear Algebra, 23:4 (2008), 285--298.






\bibitem{Li} C.K. Li, Matrices with some extremal properties, Linear Algebra Appl. 101 (1988), 255-267.

\bibitem{LTWW} C.K. Li, M.C. Tsai, Y. Wang, and N. Wong, Linear maps preserving
parallel pairs, in preparation.
Presentation at 2022 ICMAA.
https://cklixx.people.wm.edu/2022ICMAA.pdf

\bibitem{Marcus-Molys} M. Marcus, B.N. Moyls, Transformations on tensor product spaces, Pacific J. Math. 9 (1959), 1215--1221.


\bibitem{Morita} K. Morita. Analytical characterization of displacements in general Poincar\'e space. Proc. Imperial Acad. 17, No. 10 (1941), 489--494.

\bibitem{Russo} B. Russo, Trace preserving mappings of matrix algebras, Duke Math. J. \textbf{36} (1969). 297--300.

\bibitem{Thompson} R.C. Thompson, Singular values, diagonal elements and convexity,
SIAM J. AppZ. Math. 32 (1977), 39--63.



\bibitem{rais} M. Ra\"{i}s, The unitary group preserving maps (the infinite
dimensional case), Linear and Multilinear Algebra 20 (1987), no. 4, 337--345.


\bibitem{Westwick} R. Westwick, Transformations on tensor spaces, Pacific J. Math. 23 (1967) 613--620.
\end{thebibliography}
\end{document}